\documentclass{amsart}
\usepackage{localMacros}
\usepackage{nameref}
\usepackage{varioref}
\usepackage{hyperref}

\tikzset{every picture/.append style={remember picture}}

\begin{document}
\title{Vanishing Theorems for the de Rham Complex of unitary local system}
\author{Hongshan Li}
\address{
  Department of Mathematics \\
  Purdue University \\
  West Lafayette, IN 47906 \\
  U.S.A
}

\maketitle
\long\def\/*#1*/{}

\begin{abstract}
We will prove a Kodaira-Nakano type of vanishing theorem for
the logarithmic de Rham complex of unitary local system. We 
will then study the weight filtration on the logarithmic de 
Rham complex, and prove a stronger statement for the associated
graded complex.
\end{abstract}

\section*{Introduction}
Let $X$ be a smooth projective variety of dimension $n$ 
over $\complex$ and $D \subset X$
a simple normal crossing divisor.
In \cite{Deligne70},
Deligne constructed the canonical extension $E$ for any local system 
$\mathcal{L}$ defined on $U$ (over complex topology). $E$ is equipped
with a flat connection, 
\[
  \nabla: E \rightarrow E\otimes \Omega_X(\log D)
\]
and it is characterized by the 
following two properties
\begin{enumerate}
\item The flat sections of $\nabla$ coincide with $\mathcal{L}$ 
on $U$.
\item The eigenvalues of the residue of $\nabla$ lie in the strip
\[
  \{z \in \complex | 0 \le \text{Re}(z) < 1 \}
\]
\end{enumerate}

Let $V$ be a unitary local system 
on $U \definedas X - D$ (over complex topology). Let $(E, \nabla)$ 
be the canonical extension of $V$. Write $\DR_X(D, E)$ for the  
de Rham complex 
\[
  E \xrightarrow{\nabla} E\otimes\Omega_X(\log D) 
    \rightarrow \cdots 
    \xrightarrow{\nabla} E\otimes\Omega^n_X(\log D)
\]

First, we will prove a Kodaira-Nakano type of vanishing theorem
\begin{theorem}
Let $L$ be an ample line bundle on $X$, then 
\[
  H^q(X, E\otimes \Omega^p_X(\log D)\otimes L) = 0
\]
for $p + q > \dim X$.
\end{theorem}

The de Rham complex $\DR_X(D, E)$ comes with an decreasing filtration 
$F^.$ (Hodge filtration or "naive" filtration) 
and a increasing filtration $W_.$ (weight filtration).
The weight filtration $W_.$ will be defined in Section 
\ref{sect:weightfiltration}. $F^.$ and $W_.$ togethe will define
a mixed Hodge structure on $\DR_X(D, E)$. 

Then, we will prove a more refined version of the above theorem
\begin{theorem}
Let $L$ be an ample line bundle on $X$, then
\[
  H^q(X, \Gr^W_.(\Omega_X^p\otimes E)\otimes L) = 0
\]
\end{theorem}

\textbf{Acknowledgement} The author is thankful for D. Arapura for explaining
mixed Hodge theory and many key suggestions.

\section{Residue map}
\label{sect:residuemap}
In this section we will define a residue map $\Res(E)$ on the complex $\DR_X(D, E)$.
Similar to the usual residue map on the holomorphic de Rham complex,
$\komplex \Omega_X$, $\Res(E)$ will define a weight filtration on $\DR_X(D, E)$.
$\Res(E)$ has been defined and studied in \cite{timmerscheidt87}.

For $m = 1, \cdots n$, 
let $D_m$ be the union of $m$-fold intersection of components of $D$;
Let $\tilde D_m$ be the disjoint union of components of $D_m$;
Let $v_m: \tilde D_m \rightarrow X$ be the composition of the projection map onto $D_m$
and the inclusion map. $\tilde C_m \definedas \pullback v_m D_{m+1}$ 
is either empty or a normal

\begin{theorem}\cite{timmerscheidt87}
\label{thm:localDes}
\begin{enumerate}
\item $V_m := \pushforward{j}V|_{D_m - D_{m+1}}$ 
is a unitary local system 
on $D_m - D_{m+1}$.
\item There exist a unique subvectorbundle $E_m$ of 
$\pullback v_m E$ and a unique
holomorphic integrable connection $\nabla_m$ on $E_m$ with 
logarithmic poles
along $C_{m}$ such that 
\[
  \ker\nabla_{m}|_{\tilde D_m - \tilde C_m} = \inverse{v_m}V_m
\]

\item There exists a unique subvectorbundle $E_m^*$ of 
$\pullback v_m E$ with
\[
  E_m \oplus E_m^* = \pullback v_m E
\]
\end{enumerate}
\end{theorem}

\begin{proof}
All of the statements above are local. Therefore, we can assume
$X$ is a polydisk.
Write $X = \polydisk$, and let $z_i$ be the coordinate on 
$\Delta_i$. Suppose $D$ is defined by 
\[
  z_1\times\cdots\times z_s = 0
\]

\bigskip
1. 
The local system $V$ on $U$ is equivalent to an unitary 
representation
\[
  T: \pi_1(U) \rightarrow \GL(r, \complex)
\]
As $\pi_1(U)$ is abelian and $T$ is unitary, we can simultaneously
diagonalize all $T(\gamma_i$, where $\gamma_i$'s form a generating
set of $\pi_1(U)$ (see Appendix 1). Therefore, we can assume $V$
is a direct sum of rank 1 unitary local systems. Write
\[
  V = V^1\oplus\cdots \oplus V^r
\]
For each $V^i$, let $\lambda_{i,j}$ be its monodromy around $D^j$.
So $V^i$ extends to $D^j$ if and only if $\lambda_{i,j} = 1$. 

Now let $D^{j1}\cap\cdots\cap D^{jm}$ be one component of $D_m$,
and let $x \in D^{j1}\cap\cdots\cap D^{jm}$. Then, near $x$
$V_m$ is 
\[
  \bigoplus\limits_{\lambda_{i,j1}=\cdots=\lambda_{i,jm}=1}
  V^i
\]
This shows that $V_m$ is a unitary local system.

\bigskip
2. 
The uniqueness of the subvectorbundle 
$E_m$ follows from the uniqueness of
canonical connection. Therefore, we only need to show the 
existence part. 
Use the notation from part 1, 
and assume $V$ decomposes as direct sum of
rank 1 unitary local system $V^i$. 
Let $E^i$ be the canonical connection
of $V^i$. Then, it is clear that 
\[
  E_m = \bigoplus
  \limits_{\lambda_{i,j1}=\cdots=\lambda_{i, jm}=1}
  \pullback v_m E^i
\]

\bigskip
3. $E$ inheits a flat Hermitian form from $V$. 
Define $E_m^*$ as the complement
of $E_m$ with respect to this metric. 
On $\Delta$, $E_m^*$ is the direct sum of 
$\pullback v_m E^i$ not appearing in the definition
of $E_m$.
\end{proof}   

\begin{remark}
$E_m$ could have different ranks on different component
of $\tilde D_m$.
\end{remark}

For each $m \le p \le \dim D_m$, there exists a residue map
crossing divisor in $\tilde D_m$.
\[
  \Res_m: \logpform \rightarrow \pushforward{v_m}(\Omega_{\tilde D_m}^{p-m})
\]

\/*
This map is defined as follow:
Let $D_{m1}$ be one of components of $D_m$, and suppose $D_{m1}$ is the intersection of 
$D_{i1},\cdots,D_{im}$. Then, the map $\Res_m$ sends $dz_{i}/z_i$ to 1 if 
$i$ appears in $i_1, \cdots, i_m$, and $\Res_m$ sends all other 1-form to 0.
This map is well-defined independent of the chosen coordinate.
*/

$\Res_m$ commutes with exterior derivative $d$, 
making it a homomorphism
of complexes
\[
  \Res_m: \komplex{\Omega_X}(\log D) \rightarrow 
    \pushforward{v_m}\komplex{\Omega_{\tilde D_m}}(\log \tilde C_m)[-m]
\]

Consider the following variation of the residue map $\Res_m$
\begin{align*}
  \Res_m(E): \Omega_X^p(\log D)\otimes E & \xrightarrow{\Res_m\otimes\id}
    \pushforward {v_m}(\Omega_{\tilde D_m}^{p-m}(\log \tilde C_m))\otimes E \\
  & = \pushforward{v_m}(\Omega_{\tilde D_m}^{p-m}(\log \tilde C_m)\otimes 
      \pullback v_mE) \\
  & \xrightarrow{\id\otimes p_m}
    \pushforward {v_m}(\Omega_{\tilde D_m}^{p-m}(\log \tilde C_m)\otimes 
   E_m) \\
\end{align*}
where $p_m: \pullback v_m E \rightarrow \tilde E_m$
is the projection onto the $E_m$ component.

\begin{lemma}\cite{timmerscheidt87}
$\Res_m(E)\circ \nabla = \nabla_m \circ \Res_m(E)$, \emph{i.e.} $\Res_m(E)$ is 
homomorphism of complexs
\[
  \DR_X(D, E) \rightarrow 
  \pushforward {v_m} \DR_{\tilde D_m}(\tilde C_m, E_m)[-m]
\]
\end{lemma}

\section{Weight filtration on the de Rham Complex}
\label{sect:weightfiltration}
The residue map
\[
  \Res_m(E): \DR_X(D, E) \rightarrow 
  \pushforward {v_m}\DR(\tilde C_m, E_m, \nabla_m)[-m]
\]
can be used to define a weight filtration $W_.$ on $\DR_X(D, E)$
\cite{timmerscheidt87}

\begin{align*}
 & W_m(\DR_X(D, E)) = \ker\Res_{m+1}(E) & \text{ if } m \ge 0 \\
 & W_m(\DR_X(D, E)) = 0 & \text{ if } m < 0
\end{align*}

Local descriptions of $W_m(\DR_X(D, E))$ have been given in 
\cite{timmerscheidt87}. We will review them here:

Let $\Delta = \polydisk$ be a polydisk of $X$ with coordinate $z_1,\cdots,z_n$.
Suppose $D$ is defined as 
\[
  z_1 \times\cdots\times z_s = 0
\]
As in part 1 of Theorem \ref{thm:localDes}, we assume $V$ is the direct sum
of rank 1 unitary local systems on $\Delta$, and write
\[
  V = V^1\oplus \cdots \oplus V^r
\]

\begin{definition}
We say $\frac{dz_j}{z_j}$ acts on $V^i$ by identity if $\lambda_{i,j} = 1$, 
i.e. the monodromy of $V^i$ by a small circle around $D_j$ is the identity.
\end{definition}

Let $E^i$ be be canonical extension of $V^i$ on $\Delta$;

Let $\mu_i$ be a generator of $E^i$, then
\[
  \frac{dz_{j_1}}{z_{j_1}}\wedge\cdots\wedge\frac{dz_{j_k}}{z_{j_k}}
  \wedge dz_{j_{k+1}}\cdots\wedge dz_{j_p}\otimes \mu_i
\]
is in $W_m(\DR_X(D, E))$ if and only if there are at most 
$m$ log forms acting on
$V^i$ by identity.

\begin{proposition}\cite{timmerscheidt87}
\label{prop:gradedComplex}
\begin{enumerate}
\item $W_.(\DR(D,E, \nabla))$ is an increasing filtration.
\item $\Res_m(E)$ induces an isomorphism
\[
  \Gr^W_m(\DR(D, \nabla, E)) \rightarrow 
  \pushforward {v_m} (W_0(\DR_{\tilde D_m}
  (\tilde C_m, \tilde E_m))[-m])
\]
\end{enumerate}
\end{proposition}

\begin{proof}
The statements are local. We can assume $X$ is a polydisk and $V$ is a unitary
local system of rank 1. 

\bigskip
1. From the local description of $W_m(\DR_X(D, E))$, it is clear 
that $W_.$ is an increasing filtration.

\bigskip
2. Let $s$ be a section $W_m(\DR_X(D, E))$. Use the local description
above, $s$ is of the form
\[
  \omega\otimes \mu
\]
where 
\[
  \omega = 
  \frac{dz_{j_1}}{z_{j_1}}\wedge\cdots\wedge\frac{dz_{j_k}}{z_{j_k}}
  \wedge dz_{j_{k+1}}\cdots\wedge dz_{j_p}
\]
and $\omega$ has at most $m$ log 1-forms acting on $V$ by identity.
$\mu$ is a generating section of $E$.

First, we show $\Res_m(E)(s) \in W_0(\Omega_{\tilde D_m}^{p-m}(\log \tilde C_m)
\otimes E_m$. 
\[
  \Res_m(E)(s) = \Res_m(\omega)\otimes\mu_m
\]
By the construction of $\omega$, $\Res_m(\omega)$ does not have log form
$\frac{dz_{j}}{z_j}$
acting on $V$ by identity. This shows that
\[
  \Res_m(E)(s) \in W_0(\Omega^{p-m}_{\tilde D_m}(\log \tilde C_m)\otimes E_m)
\]

If $\omega_0\otimes\mu_m \in W_0(\Omega^{p-m}_{\tilde D_m}(\log \tilde C_m)\otimes E_m)$,
to get a preimage in $W_m(\Omega^p_X(\log D)\otimes E)$, simply take
\[
  \omega_m\wedge\omega_0\otimes \mu
\]
where $\omega_m$ is any $m$-form. And $\omega_m\wedge\omega_0\otimes \mu
\in W_m(\Omega_X^p(\log D)\otimes E)$ by the construction of $\omega_0$. 
This shows that 
\[
  \Res_m(E): W_m(\DR(D, E,\nabla)) \rightarrow 
  W_0(\DR(\tilde C_m, \tilde E_m, \tilde \nabla_m))
\]
is surjective.

If $\Res_m(E)(s) = 0$, that means in $\omega$, 
there are at most $m-1$ log forms
acting on $V$ by identity. This is precisely the local description of
$W_{m-1}(\Omega_X^p(\log D)\otimes E)$.
\end{proof}

\section{Mixed Hodge Structure on the de Rham Complex}
The framework for studying the mixed Hodge structure on $\DR_X(D,E)$ 
has been worked out by Deligne in \cite{Deligne71} and \cite{Deligne74}.
The analysis of the mixed Hodge structure on $\DR_X(D,E)$ 
was given by Timmerscheidt in 
\cite{timmerscheidt87}. We will give an overview about the results
from both authors. The vanishing theorem in the following section 
is a consequence of the mixed Hodge structure on $\DR_X(D,E)$.

Let $A$ denote $\integer, \rational$ or $\real$ and $A\otimes \rational$
the field $\rational$ or $\real$. 

Assume $V$ has a $A$-lattice throughout this section, i.e. there exists
a unitary local system $V_{A}$ defined over $A$ such that
\[
  V = V_{A}\otimes_{A}\complex
\]

Let $D^+(A)$(resp. $D^+(\complex))$  denote the derived category of $A$-modules
(resp. $\complex$-vector spaces)

The main result of this section is
\begin{theorem}\cite{timmerscheidt87}(Proposition 6.4)
\[
  (\derivedpushforward{j}V_A, 
  (\derivedpushforward{j}V_{A\otimes\rational}, \tau),
  (\DR_X(D, E), F, W))
\]
is an $A$-cohomological mixed Hodge complex.
\end{theorem}

For readers' sake, we included all relvant definitions involved
in the above theorem here. They can be found in \cite{Deligne74}
or \cite{Zein10}(Section 5)

\begin{definition}(Hodge Structure (HS))
A Hodge structure of weight $n$ is defined by the data:
\begin{enumerate}
\item A finitely generated abelian group $H_{\integer}$;
\item A decomposition by complex subspaces:
\[
  H_{\complex}\definedas H_{\integer}\otimes_{\integer} \complex
  = \bigoplus\limits_{p+q=n}H^{p,q}
\]
satisfying 
\[
  H^{p,q} = \overline{H^{q,p}}
\]
\end{enumerate}

\end{definition}

\begin{definition}(Hodge Complex (HC))
A Hodge $A$-complex $K$ of weight $n$ consists of 
\begin{enumerate}
\item A complex $K_A$ of $A$-modules, such that $H^k(K_A)$ is 
an $A$-module of finite type for all $k$;
\item A filtered complex $(K_{\complex}, F)$ of $\complex$-vector
spaces ;
\item Anisomorphism 
\[
  \alpha : K_A \otimes \complex \rightarrow K_{\complex}
\]
in $D^+(\complex)$;
\end{enumerate}
The following axioms must be satisfied
\begin{enumerate}
\item The spectral sequence defined by $(K_{\complex}, F)$ degenerates
at $E_1$;

\item for all $k$, the filtration $F$ on 
$H^k(K_{\complex}) \cong H^k(K_A)\otimes \complex$ defines a 
$A$-Hodge strucutre of weight $n + k$ on $H^k(K_A)$
\end{enumerate}
\end{definition}

\begin{definition}
Let $X$ be a topological space. An $A$-Cohomological Hodge Complex (CHC)
$K$ of weight $n$ on $X$, consists of:
\begin{enumerate}
\item A complex of sheaves $K_A$ of $A$-modules on $X$;
\item A filtered complex of sheaves $(K_{\complex},F)$ of 
$\complex$-vector spaces on $X$;
\item an isomorphism 
\[
  \alpha: K_A\otimes \complex \rightarrow K_{\complex}
\]
in $D^+(X, \complex)$
\end{enumerate}
Moreover, the triple $(R\Gamma(K_A), R\Gamma(K_{\complex}, F), R\Gamma(\alpha))$
is a Hodge Complex of weight $n$
\end{definition}

\begin{definition}(Mixed Hodge Complex)
An $A$-Mixed Hodge Complex (MHC) $K$ consists of:
\begin{enumerate}
\item A complex $K_A$ of $A$-modules such that $H^k(K_A)$ is an $A$-module
of finite type for all $k$;
\item A filtered complex $(K_{A\otimes \rational}, W)$ of $A\otimes\rational$-
vector spaces with an increasing filtration $W$;
\item An isomorphism $K_A\otimes \rational \cong K_{A\otimes\rational}$
in $D^+(A\otimes\rational)$;
\item A bi-filtered complex $(K_{\complex}, W, F)$ of $\complex$-vector spaces
with an increasing (resp. descreasing) filtration $W$ (resp. $F$) and an 
isomorphism:
\[
  \alpha: (K_{A\otimes\rational}, W)\otimes\complex 
  \cong (K_{\complex}, W)
\]
in $D^+F(\complex)$.
\end{enumerate}
Moreover, the following axiom needs to be satisfied:
For all n, the system consisting of
\begin{itemize}
\item the complex $\Gr^W_n(K_{A\otimes\rational})$ of $A\otimes\rational$-
vector spaces,
\item the complex $\Gr^W_n(K_{\complex}, F)$ of $\complex$-vector spaces 
with induced $F$ filtration,
\item the isomorphism 
\[
  \Gr^W_n(\alpha): \Gr^W_n(K_{A\otimes\rational})\otimes\complex
  \rightarrow \Gr^W_n(K_{\complex})
\]
\end{itemize}
is an $A\otimes\rational$-Hodge Complex of weight $n$.
\end{definition}

\begin{definition}(Cohomological Mixed Hodge Complex (CMHC))
An $A$-Cohomological Mixed Hodge Complex $K$ (CMHC) on a topological 
space $X$ consists of:
\begin{enumerate}
\item A complex of sheaves $K_A$ of sheaves of $A$-modules on $X$ 
such that $H^k(X, K_A)$ are $A$-modules of finite type;
\item A filtered complex $(K_{A\otimes\rational}, W)$ of sheaves 
of $A\otimes\rational$-vector spaces on $X$ with 
an increasing filtration $W$ and an isomorphism 
\[
  K_A\otimes\rational \cong K_{A\otimes\rational}
\]
in $D^+(X, A\otimes\rational)$;
\item A bi-filtered complex of sheaves $(K_{\complex}, W, F)$ of
$\complex$-vector spaces on $X$ with an increasing (resp. descreasing)
filtration $W$ (resp. $F$) and an isomorphism:
\[
  \alpha: (K_{A\otimes\rational}, W)\otimes\complex 
  \rightarrow (K_{\complex}, W) 
\]
in $D^+F(X, \complex)$.
\end{enumerate}
Moreover, the following axiom needs to be satisfied:
For all $n$, the system consisting of:
\begin{itemize}
\item the complex $\Gr^W_n(K_{A\otimes\rational})$ of sheaves 
of $A\otimes\rational$-vector spaces on $X$,
\item the complex $\Gr^W_n(K_{\complex}, F)$ of sheaves of 
$\complex$-vector spaces 
with induced $F$ filtration,
\item the isomorphism 
\[
  \Gr^W_n(\alpha): \Gr^W_n(K_{A\otimes\rational})\otimes\complex
  \rightarrow \Gr^W_n(K_{\complex})
\]
\end{itemize}
is an $A\otimes\rational$-Cohomological Hodge Complex of weight $n$.
\end{definition}

The following example of Cohomological Mixed Hodge Complex can be found 
in \cite{Deligne74} and \cite{Zein10}
\begin{example}
\label{ex:CMHC}
Let $X$ be a smooth projective variety over $\complex$, $D \subset X$
a simple normal crossing divisor. 
Let $U \definedas X - D$ and let 
\[
  j: U \rightarrow X
\]
be the inclusion map. 

Let $\rational_U$ be the constant sheaf with $\rational$-coefficient
on $U$. $\derivedpushforward{j}\rational_U\otimes \complex = 
\derivedpushforward{\complex_U}$ is quasi-isomorphic to 
the logarithmic de Rham complex
\[
  O_X \xrightarrow{d} \Omega_X(\log D)
  \xrightarrow{d}\cdots \xrightarrow{d}
  \Omega_X^n(\log D)
\]

For any complex $\komplex{K}$ 
of sheaves on $X$, let $\tau$ be the canonical increasing filtration
\begin{gather*}
\tau_m K^q =  
\begin{cases}
K^q & \text{ if } q < m \\
\ker{d^q} \subset K^q & \text{ if } q = m \\
0 & \text{ if } q > m
\end{cases}
\end{gather*}

See \cite{Zein10}(Corollary 6.4) for the following result:

The system consisting of
\begin{enumerate}
\item $(\derivedpushforward{j}\rational_U, \tau)$;
\item $(\komplex{\Omega_X}(\log D), W, F)$ with usual weight and Hodge
filtration $W$ and $F$;
\item The quasi-isomorphism 
\[
  (\derivedpushforward{j}\rational_U,\tau)\otimes\complex 
  \cong 
  (\komplex{\Omega_X(\log D)}, W)
\]
\end{enumerate}

is a Cohomological Mixed Hodge Complex on $X$.   
\end{example}

\section{Vanishing Theorem for the de Rham Complex}
\label{sect:vanishingThm}
we have seen in the previous section that if $V$ has a real lattice,
then 
\[
  (\derivedpushforward{j}V_A,
  (\derivedpushforward{j}V_{A\otimes\rational}, \tau),
  (\DR_X(D,E),F,W))
\]
is an $A$-cohomological mixed Hodge complex. As a result of the 
general theory developed in \cite{Deligne74}, we have

\begin{theorem}
\label{thm:spectralSequenceDeg}
Assume there is a real-valued unitary locayl system $V_{\real}$ 
defined on $U$ such that
\[
  V = V_{\real}\otimes_{\real}\complex
\]
Let $V$ and $\DR_X(D, E)$ be as above. 
The spectral sequence associated to the Hodge 
filtration on $\DR_X(D, E)$.
\[
  E^{p,q}_1 = H^q(X, \Omega_X^p(\log D)\otimes E) => 
  \hypercohomology^{p+q}(X, \DR_X(D, E))
\]
degenerates at $E_1$
\end{theorem}

If $V$ does not have an $A$-lattice with $A \subset \real$, then 
we cannot expect $\DR_X(D, E)$ to carry a mixed Hodge structure.
However, the degeneration of Hodge spectral sequence still holds
true. One can proofs of this in \cite{timmerscheidt87}. We will give
a simpler proof here:

Let $\bar{V}$ denote the conjugate of $V$, i.e. the monodromy representation
of $\bar V$ is the complex conjugate of the monodromy representation of $V$

\begin{lemma}
There exists a real unitary local system $W_{\real}$ of rank $2r$ such that
\[
  V\oplus \bar V \cong W_{\real}\otimes_{\real}\complex
\]
\end{lemma}

\begin{proof}
We will construct $W_{\real}$ locally, and show it is canonically determined
by $V$. Over a polydisk, we can assume $V$ is diagonal, and we write
\[
  V = \bigoplus_{j=1}^r V^j 
\]
where $V^i$ is a unitary local system of rank 1 with monodromy
\[
  \lambda_j = \cos\theta_j + i\sin\theta_j
\]

We will construct $W^j_{\real}$ for each $j$. The monodromy of $\bar V^j$ is
$\bar \lambda_j$ and the monodromy of $V^j\oplus \bar V^j$ is 
\[
\begin{bmatrix}
  \cos\theta_j + i\sin\theta_j & 0 \\
  0 & \cos\theta_j - i\sin\theta_j \\
\end{bmatrix}
\]
Since 
\[
\begin{bmatrix}
  \cos\theta_j + i\sin\theta_j & 0 \\
  0 & \cos\theta_j - i\sin\theta_j \\
\end{bmatrix}
\text{ and }
\begin{bmatrix}
  \cos\theta_j  & \sin\theta_j \\
  -\sin\theta_j & \cos\theta_j \\
\end{bmatrix}
\]
have the same characteristic polynomial over $\complex$, they must be
conjugate over $\complex$. Therefore, we can take $W^j$ to be
\[
\begin{bmatrix}
  \cos\theta_j  & \sin\theta_j \\
  -\sin\theta_j & \cos\theta_j \\
\end{bmatrix}
\]
Then,
\[
  W_{\real} = \bigoplus_{j=1}^r W^j
\]
\end{proof}

Now, let $V$ be any unitary local system on $X - D$, and 
let $\DR_X(D,E)$ be its de Rham complex

\begin{corollary}
The Hodge spectral sequence
\[
  E^{p,q}_1 := H^q(X, \Omega^p_X(\log D)\otimes E)
  => \hypercohomology^{p+q}(X, \DR(D,E)) 
  = H^{p+q}(X, \derivedpushforward{j}V)
\]
degenerates at $E_1$.
\end{corollary}

\begin{proof}
Direct sum and taking cohomology commutes
\end{proof}

\begin{theorem}\cite{artin73}(Corollary 3.5)
\label{thm:artinVanishing}
Suppose $U$ is an affine variety of complex dimension $n$. Then, for any 
constructible sheaf $\mathcal{L}$ on $U$
\[
  H^k(U, \mathcal{L}) = 0
\]
for $k > n$
\end{theorem}

\begin{corollary}
\label{coro:artinVanishing}
Let $V$ and $\DR_X(D, E)$ be as above. Suppose $U$ is affine, then
\[
  H^q(X, \Omega_X^p(\log D)\otimes E) = 0 
\]
for $p + q > \dim X$
\end{corollary}

\begin{lemma}
\label{lemma:1}
Suppose $B$ is a smooth divisor transversal to $D$. 
Then, there is short exact sequence 
\[
    0 \rightarrow \Omega_X^p(\log D + B)\otimes O_X(-B) 
    \xrightarrow{i} \logpform
    \xrightarrow{r} \Omega^p_B(\log D\cap B)
    \rightarrow 0
\]
where $i$ is the inclusion map, and $r$ is the restriction map.
\end{lemma}

\begin{proof}
For simplicity, we prove the case for $p=1$. 
We may also assume $X$ is affine. Let $X = \Spec A$, and let
$f_1,\cdots,f_s$ be the regular sequence corresponding to $D$, and let
$b$ be the defining equation of $B$. 

The basis of 
$\Omega_X^1(\log D + B)\otimes O_X(-B)$ as an $A$-module is
\[
  \frac{df_1}{f_1}\otimes b,\cdots, \frac{df_s}{f_s}\otimes b, 
  \frac{db}{b}\otimes b
\]

The basis of $\Omega_X(\log D)$ as an $A$-module is
\[
  \frac{df_1}{f_1},\cdots, \frac{df_s}{f_s}
\]

The basis of $\Omega_B(\log D\cap B)$ as an $\frac{A}{b}$-module is
\[
  \frac{df_1}{f_1},\cdots, \frac{df_s}{f_s}
\]
where by abuse of notation $f_i$ are regarded as their image in $\frac{A}{b}$.

Then, it is clear how to define $i$ and $r$ show that the above sequence 
is exact
\end{proof}

\begin{lemma}
\label{lemma:2}
Suppose $B$ is a smooth divisor transversal to $D$. Then, 
$E_B:= E\otimes O_B$ is the canonical extension of $V_B: = V|_{B - B\cap D}$.
\end{lemma}

\begin{proof}
The statement is local, therefore we may assume $X$ is a polydisk
\[
  \polydisk
\]
such that the analytic coordinate of $\Delta_i$, for $i=1,\cdots,s$, are 
defining equation of $D_i$, and the analytic coordinate of $\Delta_n$ is
the defining equation of $B$.

First, we study $V_B$ by computing its monodromy representation: 

Let $T: \pi_1(X-D, x) \rightarrow \GL(r, \complex)$ be the monodromy
representation of $V$. For each generator $\gamma_i$ of $\pi_1(X-D, x)$,
let $\Gamma_i = T(\gamma_i)$. As $\Gamma_i$ are commuting and unitary,
we can use one matrix to diagolize all of them. Therefore, we can assume
all $\Gamma_i$ are diagonal matrices. Moreover, as $V$ is undefined 
only on $D$, so for each $i$, $\Gamma_i^{jj} = 1$, for $j=s+1,\cdots,n$.

Now, $B = \Delta_1\times\cdots\times\Delta_{n-1}$,
and the monodromy reprentation of $V|_{B - B\cap D}$ is given by
\[
  \pi_1(B-B\cap D) \xrightarrow{i} \pi_1(X-D) \xrightarrow{T}
  \GL(r, \complex)
\]
where $i$ is the natural inclusion map. It is clear that one can choose
the basis of $\pi_1(B - B\cap D)$ and $\pi_1(X - D)$ such that $i$ can be
realized as the identity map. Therefore, the monodromy representations
of $V_{B - B\cap D}$ are also $\Gamma_i$, for $i=1,\cdots,s$.

To show $E|_B$ is the canonical extension of $V_{B - B\cap D}$, we compute
the connection matrix of $E|_B$ and relate it to the monodromy representations
of $V|_{B - B\cap D}$. 

One can assume $E$ is trivial over $X$. Choose a local frame of $V$ on $X$, 
and use it as a trivialization of $E$. With respect to this trivialization, the
connection $\nabla$ can be realized as 
\[
  d + N_1\frac{dz_1}{z_1} + \cdots +  N_s\frac{dz_s}{z_s}
\]
where $N_1,\cdots,N_s$ are commuting matrices with eigenvalues in the
stripe 
\[
  \{z \in \complex | 0 \le \text{Re}z < 1 \}
\]
such that $e^{-2\pi iN_i} = \Gamma_i$.

Now, restrict $E$ to $B$, we see that the connection $\nabla|_B$ can still
be realized as
\[
  d + N_1\frac{dz_1}{z_1} + \cdots +  N_s\frac{dz_s}{z_s}
\]
As monodromy representations of $V_{B - B\cap D}$ are $\Gamma_i$,
it follows that $E|_B$ is the canonical extension of $V_{B - B\cap D}$.
\end{proof}

\begin{theorem}
\label{thm:unitaryVanishing}
Suppose $L$ is very ample on $X$. Then
\[
  H^q(X, E\otimes\logpform\otimes L) = 0
\]
for $p + q > \dim X$
\end{theorem}

\begin{proof}
Let $B$ be a smooth divisor transversal to $D$ such that $L \cong O_X(B)$.
By Lemma \ref{lemma:1} we have the following exact sequence
\[
  0 \rightarrow \Omega_X^p(\log D + B) 
    \xrightarrow{i} \logpform\otimes O_X(B)
    \xrightarrow{r} \Omega^p_B(\log D\cap B)\otimes O_X(B)
    \rightarrow 0
\]

Tensor it by $E$ and take the cohomology sequence, we get:
\begin{align*}
  \cdots & H^q(X, \Omega^p_X(\log D + B)\otimes E) 
  \rightarrow H^q(X, \Omega^p_X(\log D)\otimes O_X(B)\otimes E) \\
  \rightarrow &H^q(X, \Omega^p_B(\log B\cap D)\otimes O_X(B)\otimes E)
  \cdots
\end{align*}

Therefore, to prove the theorem, it is enough to show

\textbf{Claim 1:} $H^q(X, \Omega^p_X(\log D + B)\otimes E) = 0$

\textbf{Claim 2:} $H^q(X, \Omega^p_B(\log B\cap D)\otimes O_X(B)\otimes E) = 0$
for $p + q > \dim X$.

\smallskip
\textbf{Proof of claim 1:}
Consider the maps
\[
  X - (B+D)\xrightarrow{f} X - B \xrightarrow{h}X
\]
\renewcommand{\L}{\mathcal{L}}

Let $V^o$ be the restriction of $V$ on $X-(B+D)$. 
The complex $\DR(D+B,E,\nabla)$ is quasi-isomorphic to 
$\mathbb{R}\pushforward {(h\circ f)}V^o$. Therefore,
\[
  H^k(X-(B+D), V^o) = 
  \hypercohomology^k(X, \DR(D+B,E,\nabla)
\]
The claim then follows from Corollary \ref{coro:artinVanishing}.

\textbf{End of Proof}

\smallskip
Claim 2 follows from 
induction on the dimension of the variety.

\bigskip
Now to finish the proof, it remains to show the base case of Claim 2.
One may assume now that $X$ is a smooth projective curve over $\complex$,

We need to show that 
\[
  H^1(X, \Omega_X(\log D)\otimes E\otimes L) = 0
\]
But for the curve case, 
$\Omega_X(\log D)\otimes O_X(B) = \Omega_X(\log D + B)$.
So the result follows again from Theorem \ref{coro:artinVanishing}
\end{proof}

Now suppose $L$ is any ample line bundle. Let $m$ be an integer such that
$\sheaftensor{L}{m}$ is very ample. Take a smooth divisor $B$ transversal 
to $D$ such that $\sheaftensor{L}{m} \cong O_X(B)$. Let $\varphi$ be the
local equation of $B$ on some affine open set, and let $\pi:X^{\prime} \rightarrow X$ 
be the normalization of $X$ in $\complex(X)(\varphi^{\frac{1}{m}})$.  

\begin{proposition}
Let $\pi: X^{\prime} \rightarrow X$, $B$ and $L$ be as above
\begin{enumerate}
\item $X^{\prime}$ is smooth.
\item $\pullback{\pi}B = m\tilde{B}$, where $\tilde{B} = (\pullback{\pi}B)_{red}$.
\item $D^{\prime}:=\pullback{\pi}D$ is a normal crossing divisor on $X^{\prime}$.
\item $\tilde{B}$ is transversal to $\pullback{\pi}D$.
\item $\pullback{\pi}\Omega^p_X(\log D) = 
  \Omega^p_{X^{\prime}}(\log D^{\prime})$.
\item $\pullback{\pi}E$ is the canonical extension of $\inverse{\pi}V$.
\end{enumerate}
\end{proposition}

\begin{proof}
\newcommand{\upRing}{\frac{A_i[Y]}{(Y^m - f_i)}}
1. We will construct $X^{\prime}$ by constructing its affine cover and 
specefiying the gluing morphisms. Let $U_i = \Spec A_i$ be an affine cover
of $X$, and let $f_i$ be the defining equation of $D$ in $A_i$. 

For each $A_i$, $\frac{A_i[Y]}{(Y^m - f_i)}$ is integrally closed in 
$\complex(X)(f_i^{1/m})$. Therefore, 
\[
  U_i^{\prime}:= \Spec \frac{A_i[Y]}{(Y^m - f_i)}
\]
is the normalization of $U_i$ in $\complex(X)(f_i^{1/m})$

The same morphisms used to glue $U_i$ into $X$ can be used to glue
$U_i^{\prime}$ into $X^{\prime}$. Therefore, to show $X^{\prime}$ 
is smooth, it is enough to show $\frac{A_i[Y]}{(Y^m - f_i)}$ is 
a regular ring. 
\bigskip

2. The local defining equation of $\tilde{B}$ is $Y$, and $\pullback{\pi}(f_i) = Y^m$

\bigskip
3. To see this, we describe $\pullback{\pi}D$ in $\pullback{\pi}U$ for
any polydisk $U = \Delta_1\times\cdots\times\Delta_n$. 
If $B \cap U \neq \emptySet$, then construct $\Delta_i$ such that 
defining equation of $D_i$, for $i = 1,\cdots,s$, are coordinates
of $D_i$, for $i=1,\cdots,s$; and the defining equation of $B$
is the coordinate of $D_n$. Then, 
\[
  \pullback{\pi}U = \Delta_1\times\Delta_1\cdots\Delta_{n-1}\times \Sigma^m
\]
where $\Sigma^m$ is the $m$-sheeted cover over a complex disk branched over
the origin. In this case, $\pullback{\pi}D$ is still defined by 
$z_1\times z_2\times\cdots z_s$.

If $B \cap U = \emptySet$, then $\pullback{\pi}U$ is etale over $U$.
Therefore, $\pullback{\pi}D$ is etale over $D$. So $\pullback{\pi}D$ is
again a simple normal crossing divisor.   

\bigskip
4.This is clear from the case 1 of part 3.

\bigskip
5. Straighforward computation.
\bigskip
6. We compute the monodromy representation of $\inverse{\pi}V$ first: 

let $T: \pi_1(U-D, x) \rightarrow \GL(r, \complex)$ be the representation corresponding
to the local system $V$.

\textbf{Case 1:} 
Suppose $x \notin B$, then $\inverse{\pi}(U)$ is etale over $U$. Let $U^{\prime}$ be
an component of $\inverse{\pi}(U)$, and let $x^{\prime} \in U^{\prime}$ be a preimage 
of $x$.
Then,
\[
  T^{\prime} : \pi_1(U^{\prime} - D^{\prime}, x^{\prime}) 
  \xrightarrow{\pushforward{\pi}} \pi_1(U - D, x) 
  \xrightarrow{T} \GL(r, \complex)
\] 
is the representation corresponding to $\inverse{\pi}V$.

\textbf{Case 2:} Suppose $x \in B$, then use the description from part 3, 
we know that 
\[
  \inverse{\pi}U = \Delta_1\times \Delta_2 \times \cdots \times \Sigma^m
\]

In both cases,
$\inverse{\pi}U - D^{\prime}$ is homotopic to $S_1\times S_2\times\cdots\times S_s$
So we can define generators of $\pi_1(U^{\prime} - D^{\prime}, x^{\prime})$
and $\pi_1(U - D, x)$ such that $\pushforward{\pi}$ is the identity map. 

To show $\pullback{\pi}E$ is the canonical extension of $\inverse{\pi}V$,
we only need to compute the connection matrix of $\pullback{\pi}E$ and relate
it to the monodromies of $\inverse{\pi}V$:

Let $\gamma_i$ be a small circle around $D_i$, and let $\Gamma_i$ be the monodromy 
$T(\gamma_i)$. 
As
\[
  \pushforward{\pi}: \pi_1(U^{\prime} - D^{\prime}, x^{\prime})
  \rightarrow \pi_1(U - D, x)
\]
is the identity map, $\Gamma_i$ are also
the monodromy representations of $\inverse{\pi}V$.
Next, we compute the connection matrix of $E$. Let $U$ be small enough
so that $E$ is trivial over it. Choose a local frame of $V$, and use it
as a trivialization of $E$. With respect to this trivialization, the
connection $\nabla$ can be realized as 
\[
  d + N_1\frac{dz_1}{z_1} + \cdots +  N_s\frac{dz_s}{z_s}
\]
where $N_1,\cdots,N_s$ are commuting matrices with eigenvalue in the
stripe 
\[
  \{z \in \complex | 0 \le \text{Re}z < 1 \}
\]
such that $e^{-2\pi iN_i} = \Gamma_i$.

As $\pullback{\pi}z_i = z_i$, for $i=1,\cdots,s$, we see that the 
$\pullback{\pi}\nabla$ over $\inverse{\pi}U$ can be realized as:
\[
  d + N_1\frac{dz_1}{z_1} + \cdots +  N_s\frac{dz_s}{z_s}
\]
This shows that $\pullback{\pi}E$ is the canonical extension of $\inverse{\pi}V$.

\end{proof}

\begin{corollary}
\label{coro:unitaryVanishing}
For any ample line bundle $L$ on $X$,
\[
  H^q(X, E\otimes\logpform\otimes L) = 0
\]
for $p + q > \dim X$
\end{corollary}

\begin{proof}
Let $m$, $B$ and $\pi: X^{\prime} \rightarrow X$ be as above. By 
Theorem \ref{thm:unitaryVanishing} 
\[
  H^q(X^{\prime}, \pullback{\pi}(E\otimes\logpform\otimes L)) = 0
\]
for $p + q > \dim X^{\prime} = \dim X$.

$\pi: X^{\prime} \rightarrow X$ is a finite morphism, so for $i > 0$,
$R^i\pushforward {\pi}\mathscr{F} = 0$ for any coherent sheaf 
$\mathscr{F}$ on $X^{\prime}$. This implies
\begin{align*}
   H^q(X^{\prime}, \pullback{\pi}(E\otimes\logpform\otimes L)) 
  & = H^q(X, \pushforward{\pi}(\pullback{\pi}(E\otimes\logpform\otimes L))) \\
  & = H^q(X, \pushforward{\pi}(O_Y)\otimes E\otimes\logpform\otimes L) \\
  & = 0
\end{align*}
for $p + q > \dim X$. The second equality follows from the projection formula.

As $\pushforward{\pi}(O_Y) \cong \bigoplus\limits_{i=0}^{m-1} 
O_X(-\sheaftensor{L}{i})$, the result follows.
\end{proof}

\section{Partial weight filtration}
In the previous section, we proved the vanishing theorem for
the complex 
\[
  \DR_X(D, E)\otimes O_X(B)
\]
where $B$ is a smooth very ample divisor transversal to $D$. 
The intermediate step for the proof is the vanishing theorem for
the complex 
\[
  \DR(D+B, E)
\]

In this section,
we define a partial weight filtration on the 
The complex 
\[
  \DR_X(D+B, E)
\]
It is a more refined weight filtration than the one
defined in Section \ref{sect:weightfiltration}, and it will be used to prove
the vanishing theorem for the graded complex
\[
  \Gr^W_. \DR_X(D, E)
\]

For simplicity, suppose $V$ is 
a rank 1 unitary local system. We will define partial weight filtration by
giving local description of forms. 
Then, we will show it is a well-defined global notion after
Theorem \ref{thm:partialWeightFiltration}.
Let $\mu$ be a section of $E$. Recall that
$W_m\Omega^p_X(\log D)\otimes E$ consists of sections of the form
\[
  \omega\otimes\mu
\]
where $\omega$ can be written as
\[
  \frac{dz_{j_1}}{z_{j_1}}\wedge\frac{dz_{j_2}}{z_{j_2}}\wedge
  \cdots\frac{dz_{j_k}}{z_{j_k}}\wedge dz_{j_{k+1}} \wedge\cdots
  \wedge dz_{j_p}
\]
Moreover, $\omega$ has at most $m$ log forms acting on $V$ by identity.

Now let $W_m^{D^1}\DR(D,E)$ be the set of form that can be written as
$\omega\otimes\mu$, where $\omega$ can be written as
\[
  \frac{dz_{j_1}}{z_{j_1}}\wedge\frac{dz_{j_2}}{z_{j_2}}\wedge
  \cdots\frac{dz_{j_k}}{z_{j_k}}\wedge dz_{j_{k+1}} \wedge\cdots
  \wedge dz_{j_p}
\]
Moreover, let $g$ be the cardinality of the following set
\[
  \{\text{log forms in $\omega$ acting on $V$ by identity}\} \cap 
  \{\frac{dz_2}{z_2},\cdots,\frac{dz_s}{z_s}\}
\]
Then $g \le m$.

To generalize, $W_m^{D^{i_1}+\cdots D^{i_l}}\DR(D,E)$ is the set of
forms that can be written as $\omega\otimes\mu$, where $\omega$ can be written
as
\[
  \frac{dz_{j_1}}{z_{j_1}}\wedge\frac{dz_{j_2}}{z_{j_2}}\wedge
  \cdots\frac{dz_{j_k}}{z_{j_k}}\wedge dz_{j_{k+1}} \wedge\cdots
  \wedge dz_{j_p}
\]
Moreover, let $g$ be the cardinality of the following set
\[
  \{\text{log forms in $\omega$ acting on $V$ by identity}\} \cap 
  (\{\frac{dz_1}{z_1},\cdots, \frac{dz_s}{z_s}\} - 
  \{\frac{dz_{i_1}}{z_{i_1}},\cdots \frac{dz_{i_l}}{z_{i_l}} \})
\]

Write $T$ for $D + B$. Let 
$T_2$ be the union of 2-fold intersections of components of $T$. 

Let $v_1: \tilde D_1 \rightarrow D$ be the normalization map, 
i.e. $\tilde D_1$
is the disjoint union of components of $D$. 
Let $F_1$ be $\pullback{v_1}T_2$.
Then, $F_1$ is a normal crossing divisor in $D_1$.
  
We have seen in Section \ref{sect:residuemap} that 
the restiction of $\pushforward jV$ on $D_1 - T_2$ is a unitary local system, 
denote it by $V_1$;
and let $E_1$ be the subbundle of $\pullback v_1 E$ which is the canonical
extension of $V_1$.

\begin{proposition}
\label{prop:pwfExactSequence}
There is an exact sequence 
\[
  0 \rightarrow W^B_0\DR(D+B,E,\nabla) \xrightarrow{i} 
    \DR(D+B,E,\nabla) \xrightarrow{res} \pushforward {v_1}\DR(F_1,E_1,\nabla_1)
    \rightarrow 0
\] 
where $i$ is the inclusion map, and $res$ is the residue map.
\end{proposition}

\begin{proof}
Suppose for simplicity $D$ is smooth, i.e. $D$ has only one component.
Also, suppose $V$ is a unitary local system of rank 1. Let $\mu$ be a 
local section of $E$. 

Let $z_1$ be the local equation of $D$. Suppose $\frac{dz_1}{z_1}$ acts
on $V$ by identity, then $V$ extends to a unitary local system on 
$D - D\cap B$. In this case, $D_1 = D$, and $F_1 = D\cap B$. 
Let $z_n$ be the local equation for $B$. Then, locally over a polydisk

\smallskip
1. $W^B_0\Omega^p_X(\log D+B)\otimes E$ is generated by sections of the form
\[
  \frac{dz_n}{z_n}\wedge\omega\otimes \mu
\]
where $\omega \in \Omega^{p-1}_X$.

\smallskip
2. $\Omega^p_X(\log D+B)$ is generated by sections of the form
\[
  \frac{dz_n}{z_n}\wedge\omega\otimes\mu
\]
where $\omega \in \Omega^{p-1}_X(\log D)$.

\smallskip
3. $\Omega^{p-1}_{D_1}(\log F_1)\otimes E_1$ is generated by sections
of the form
\[
  \frac{dz_n}{z_n}\wedge\omega\otimes\mu_1
\]
where $\omega \in \Omega_{D_1}^{p-2}(\log F_1)$.

Use the local description, it is clear that the sequence is exact.

\end{proof}

\begin{theorem}
\label{thm:partialWeightFiltration}
Let $(E_B,\nabla_B)$ be the restriction of $(E,\nabla)$ on $B$, and let 
$\DR(B\cap D, E_B,\nabla_B)$ be the complex
\[
  0 \rightarrow E_B \rightarrow \Omega_B^1(B\cap D)\otimes E_B
  \cdots
\]
then there is an exact sequence of complexes
\begin{align*}
    0 
    \rightarrow W^B_m\DR_X(D+B, E) 
    \xrightarrow{i} W_m\DR_X(D, E)\otimes O_X(B) \\
    \xrightarrow{r} W_m\DR_B(D\cap B, E_B)\otimes O_X(B)
    \rightarrow 0 
\end{align*}
$i$ is the inclusion map, and $r$ is the restriction map.
\end{theorem}

\begin{proof}
For simplicity, we assume $E$ has rank 1. The statement is local, so we 
work on a polydisk, and we use the notation from above. Let $\mu$ be 
a generating section of $E$, then

\bigskip
1. $W^B_m\DR_X(D+B,E)\otimes O_X(-B)$ is generated by
\[
  \omega\otimes\mu\otimes z_n
\]
where $\omega \in \Omega^p_X(\log D + B)$ is a $p$-form that has at most
$m$ log forms coming from 
\[
  \{\frac{dz_1}{z_1},\cdots,\frac{dz_s}{z_s}\}
\]
acting on $E$ by identity.

\smallskip
2. $W_m\DR_X(D,E)$ is generated by
\[
  \omega\otimes\mu
\]
where $\omega \in \Omega^p_X(\log D)$ is a $p$-form that has at most $m$
log forms acting on $E$ by identity.

\smallskip
3. $W_m\DR_B(D\cap B,E_B)$ is generated by 
\[
  \omega\otimes \mu
\]
where $\omega \in \Omega_B^p(\log B\cap D)$ is a $p$-form that
has at most $m$ log forms acting on $\mu_B$ by identity.

The map $i$ is the natural inclusion map, 
i.e. $\frac{dz_n}{z_n}\otimes z_n \mapsto dz_n$;
The map $r$ is the restriction on $B$.

\end{proof}

The above theorem also gives a description of 
\[
  W_m^B\DR_X(D+B,E)
\]
as the kernel of the restriction map
\[
  r: \DR_X(D+B,E)\otimes O_X(B) \rightarrow 
     W_m\DR_B(B\cap D, E_B)\otimes O_X(B)
\]
It means that $W_m\DR_X(D+B,E)$ is indeed globally well-defined.

\section{Mixed Hodge Structure on the Complex $W^B_0\DR(D+B,E,\nabla)$}
\newcommand{\hodgelattice}{\derivedpushforward h \pushforward f V^o_{\real}}
\newcommand{\canonicalfiltration}{
  (\derivedpushforward h \pushforward f V^o_{\real}, \tau)
}
\newcommand{\bifilteredcomplex}{
  (\bcomplex, F^., W_.)
}  
\newcommand{\cmhc}{(\hodgelattice, \canonicalfiltration, \bifilteredcomplex)}

Throughout this section, we assume the unitary local system $V$ has a real
lattice $V_{\real}$ such that 
\[
  V = V_{\real}\otimes \complex
\]

We will study the mixed Hodge structure on 
the complex
\[
  W^B_0\DR_X(D+B,E)
\]

Consider the maps
\[
  X - (D + B) \xrightarrow{f} X - B \xrightarrow{h} X
\]
Write $V^o$ (resp. $V^o_{\real}$) for the restriction of $V$ (resp. $V_{\real}$)
on $X - (D+B)$.

Let $\tau$ be the canonical filtration on 
$\derivedpushforward h \pushforward f V^o_{\real}$; let $W$ be the 
increasing filtration on $W^B_0\DR_X(D+B,E)$ defined as

\begin{gather*}
W_m \bcomplex = 
\begin{cases}
    0 & \text{ if } m < 0 \\
    W_0\DR_X(D+B,E) & \text{ if } m = 0\\
    \bcomplex & \text{ if } m > 0
\end{cases}
\end{gather*}

The main result of this section is 
\begin{theorem}
\label{thm:MHSPartialWeightFiltration}
\[
  \cmhc
\]
is a $\real$-cohomological mixed Hodge complex.
\end{theorem}

\begin{proposition}
$\mathbb{R}\pushforward h(\pushforward{f}V^o)$ is quasi-isomoprhic to 
\[
  W^B_0\DR_X(D + B, E)
\]
\end{proposition}

\begin{proof}
The statement is local, so we can assume $X$ is a polydisk. 
For the basic case, one can assume 
$V$ is of rank 1, $D$ has two components $D^1$ and $D^2$
such that the monodromy of $V$ around $D^1$ is trivial, and the monodromy of $V$
around $D^2$ is nontrivial. 

Let $Y = X - B$. Then, $\komplex \Omega_Y(\log D^2)\otimes \pullback h E$ is 
a resolution of $\pushforward{f}V^o$(see \cite{timmerscheidt87}). 

Let $g: Y - D^2 \rightarrow Y$ be the inclusion map. 
By a theorem of Griffith\cite{Griffith69} and Deligne\cite{Deligne71}, 
the inclusion map
\[
  i: \komplex \Omega_Y(\log D^2) 
  \rightarrow \pushforward g\komplex  {\mathscr A}_{Y-D^2}
\]
is a quasi-isomorphism. Therefore, $\pushforward f V$ is quasi-isomorphic to
\[
  \pushforward g \komplex {\mathscr{A}}_{Y-D^2}
\]

As $\pushforward g \komplex {\mathscr{A}}_{Y-D^2}$ 
is a complex of flasque sheaves,
$\mathbb{R}\pushforward h \pushforward f V$ is quasi-isomorphic to 
\[
  \pushforward h \pushforward g \komplex {\mathscr{A}}_{Y-D^2}
\]

Now, 
\[
  W^B_0 \komplex \Omega_X(\log D + B)\otimes E 
  = \komplex \Omega_X(\log D^2 + B) \otimes E
\]

But as we have seen the complex $\komplex \Omega_X(\log D^2 + B)$ 
is quasi-isomorphic to
\[
  \pushforward{(h\circ g)} \komplex {\mathscr{A}}_{Y-D^2}
\]
So the result for the basic case follows.

Now, let $V$ be of rank $r$.
For each $i=1,2$, let $\Gamma_i$ be the monodromy of $V$ around $D^i$.
As $V$ is unitary, we can simultaneously diagonalize all $\Gamma_1$ and $\Gamma_2$.
Therefore, we can assume $V$ is the direct sum of two rank 1 unitary local systems.
As $\derivedpushforward{h}$ and $\pushforward{f}$ commutes with direct sum. 
The result follows.

Now, let $V$ be of rank 1 and let $D^1,\cdots,D^s$ be components of $D$. 
Now let $D_1$ be the subdivisor of $D$ over which $V$ has identity monodromy;
and let $D_2$ be the subdivisor of $D$ over which $V$ has nontrivial monodromy.
Then, the result follows after the same steps in the basic case. 
\end{proof}

\begin{proposition}
The inclusion map 
\[
  i: (W^B_0\DR_X(D+B,E), \tau_.) \rightarrow (W^B_0\DR_X(D+B,E),W)
\]
is a quasi-isomorphism of filtered complexes.
\end{proposition}

\begin{proof}
This is again a local statement, so we can assume $X$ is a polydisk
and $V$ is of rank 1.
We need to show that the induced maps of $i$
\[
  H^k(i): H^k(\Gr^{\tau}_m \bcomplex) \rightarrow H^k(\Gr^{W}_m \bcomplex)
\]
are isomorphisms.

\begin{align*}
  H^k(\Gr^{\tau}_m W^B_0\DR_X(D+B,E)) = 
\begin{cases}
  H^m(W^B_0\DR_X(D+B,E)) & \text{ if } i = m \\
  0 & \text{ otherwise}
\end{cases}
\end{align*}

\textbf{Claim 1} If $m > 1$, then $H^m(\bcomplex) = 0$.

\textbf{Proof of Claim 1} We have a short exact sequence of complexes
\[
  0 \rightarrow W_0\DR_X(D+B,E) \rightarrow W^B_0 \DR_X(D+B,E)
    \xrightarrow{res} W_0\DR_B(B\cap D, E_B)[-1] \rightarrow 0
\]
where the $\DR(D\cap B, E_B, \nabla_B)$ is the complex
\[
  \cdots \rightarrow \Omega^m_B(\log B\cap D)\otimes E_B
  \xrightarrow{\nabla_B} \Omega^{m+1}_B(\log B \cap D)\otimes E_B 
  \rightarrow \cdots
\]
and the map $res$ is the residue map.

Taking cohomology, we get

\begin{align*}
  \cdots 
  &\rightarrow H^k(W_0\DR_X(D+B,E)) 
    \rightarrow H^k(W^B_0\DR_X(D+B,E)) \\
  & \rightarrow H^{k-1}(W_0\DR_B(B\cap D, E_B)) \rightarrow \cdots
\end{align*}

Timmerscheidt proved in the Appendix D of \cite{esnault-viehweg} that 
$W_0\DR_X(D+B,E)$ is a resolution of $(h\circ f)_*V$. Therefore,
$\bcomplex$ is exact. Likewise, 
\[
  W_0\DR_B(B\cap D, E_B)
\]
is also exact. 

So the conlusion follows.

\textbf{End of Proof}

The above proof also shows that 

\begin{gather*}
  H^k(\Gr^W_1 \bcomplex) = 
\begin{cases}
  H^1(\Gr^W_1 \bcomplex)  & \text{ if } k = 1\\
  0 & \text{ if } k > 1
\end{cases}
\end{gather*}

\begin{gather*}
  H^k(\Gr^W_0 \bcomplex) = 
\begin{cases}
  H^0(W_0\DR(D+B,E))  & \text{ if } k = 0 \\
  0 & \text{ if } k > 0 
\end{cases}
\end{gather*} 

Therefore, to prove 
\[
  i : (W^B_0\DR_X(D+B,E),\tau) 
  \rightarrow (W^B_0\DR_X(D+B,E), W)
\]
is a quasi-isomorphism of filtered complexes, it remains to prove that
both
\[
  H^0(i) : H^0(\Gr^{\tau}_0\bcomplex)
  \rightarrow H^0(\Gr^W_0 \bcomplex)
\]
and 
\[
  H^1(i) : H^1(\Gr^{\tau}_1\bcomplex)
  \rightarrow H^1(\Gr^W_1 \bcomplex)
\]
are isomorphisms.

Now,
\[
  H^0(\Gr^{\tau}_0\bcomplex) = 
  \ker (E \xrightarrow{\nabla} W^B_0(\Omega^1_X(\log D + B)\otimes E)
\]
and 
\[
  H^0(\Gr^W_0 \bcomplex) = 
  \ker (E \xrightarrow{\nabla} W_0(\Omega^1_X(\log D + B)\otimes E)
\]
It is clear that the map $H^0(i)$ is an isomorphism.

To simplify notations, write $\komplex K$ for $\bcomplex$, 
from the proof of Claim 1, we have a commutative diagram 

\[
\begin{tikzcd}
  & H^1(\Gr^{\tau}_1 \komplex K) \arrow{r}{H^1(i)} \arrow{d}
  & H^1(\Gr^W_1 \komplex K) \arrow{r}{res} 
  & H^1(W_0 \DR_B(B\cap D, E_B)[-1]) \arrow{d} \\
  & H^1(\komplex K) \arrow{rr}{res}
  && H^1(W_0 \DR_B(B\cap D, E_B)[-1])
\end{tikzcd}
\] 
and the residue map on the second row is an isomorphism. 
As the residue map on the first row is an isomorphism (even on the 
complex level), we see that the map $H^1(i)$ is an isomorphism.

\end{proof}

For reader's sake, we restate the main theorem of this Section:
\begin{theorem}
\[
  \cmhc
\]
is a cohomological mixed $\real$-Hodge complex
\end{theorem}

\begin{proof}
The quasi-isomorphism 
\[
  \canonicalfiltration \otimes \complex 
  \rightarrow 
  (\bcomplex, W_.)
\]
was proved in the previous proposition.

It remains to show
\[
  (\Gr^{\tau}_m \hodgelattice,
  (\Gr^{W}_m\bcomplex, F))
\]
is a cohomological $\real$-complex of weight $m$, i.e. the Hodge 
spectral sequence of $(\Gr^{W}_m\bcomplex, F)$ degenerates at $E_1$, and the 
induced filtration on
\[
  \hypercohomology^k(X, \Gr^W_m\bcomplex) = 
  \hypercohomology^k(X, \Gr^{\tau}_m\hodgelattice)\otimes \complex
\]
defines a pure $\real$-Hodge structure of weight $k+m$ on 
\[
  \hypercohomology^k(X, \Gr^{\tau}_m \hodgelattice)
\]
i.e. the induced filtration $F$ on $\hypercohomology^k(X, \Gr^{W}_m\bcomplex)$
is $m+k$ opposed to its conjugate.

For $m > 1$, all $\Gr^W_m\bcomplex$ are 0, so we only need to show the case for
$m =0, 1$.

For $m = 0$, 
\[
  (\Gr^W_m\bcomplex, F) = (W_0\DR_X(D+B,E), F)
\]
Timmerscheidt showed that it is a cohomological $\real$-complex
of weight $0$ in \cite{esnault-viehweg}(Appendix D).

For $m = 1$, we have seen that
\[
  \Gr^W_1\bcomplex \cong W_0\DR(B\cap D, E_B, \nabla_B)[-1]
\]
Let $F$ be the induced Hodge filtration on $\Gr^W_1\bcomplex$, and let
$F_B$ be the usual Hodge filtration on $W_0\DR(B\cap D, E_B, \nabla_B)$. 
let $\bar F$ and $\bar F_B$ be their conjugates. 

\renewcommand{\c}{\hypercohomology^k(X, \Gr^W_1\bcomplex)}
To show $F$ and $\bar F$ are $k+1$ opposed on 
$\hypercohomology^k(X, \Gr^W_1\bcomplex)$, we show that
\[
  \Gr^{\bar F}_q\Gr^F_p\c = 0 \text{ if } p + q \neq k + 1
\]
As $\Gr^W_1\bcomplex \cong W_0\DR_B(B\cap D, E_B)[-1]$, 
\begin{align*}
  \Gr^F_p \c & = \Gr^{F_B}_{p-1}
                 \hypercohomology^{k-1}(B, W_0\DR_B(B\cap D, E_B)) \\
  \Gr^{\bar F}_q \c & = \Gr^{\bar F}_{q-1}
                 \hypercohomology^{k-1}(B, W_0\DR_B(B\cap D, E_B)) \\
\end{align*}
Therefore, $\Gr^{\bar F}_q\Gr^F_p\c = 0$ if $p-1 + q - 1 \neq k - 1$.

The $E_1$-degeneration of $(\Gr^W_1\bcomplex, F)$ follows from the $E_1$-degneration
of $(W_0\DR(B\cap D, E_B, \nabla_B), F_B)$.

\end{proof}

So far, we have shown that if $V$ has a real lattice $V_{\real}$, i.e.
\[
  V = V_{\real}\otimes_{\real}\complex
\]
Then, the Hodge spectral sequence 
\[
  E_1^{p,q} = H^q(X, W^B_0\Omega^p_X(\log D + B)\otimes E)
  => \hypercohomology^{p+q}(X, \bcomplex)
\]
degenerates at $E_1$. 

Now, consider the case when $V$ does not have a real-lattice.

\section{Vanishing Theorem for the complex $\Gr^W_.\DR_X(D,E)$}
Now, let $V$ be any unitary local system over $\complex$. We have 
seen in Section \ref{sect:vanishingThm} that even if 
$V$ does not have a real lattice, the spectral sequence of 
$(\DR_X(D, E), F)$ still have $E_1$-degeneration. Similarly, we have

\begin{lemma}
\label{lemma:spectralSequenceDegWB0}
Let $B$ be a smooth divisor transversal to $D$, then
The spectral sequence of $(W^B_0\DR(D+B, E), F)$:
\[
  E^{p,q}_1 = H^q(X, W^B_0(\Omega_X^p(\log D + B)))
  => \hypercohomology^{p+q}(X, W^B_0(\DR_X(D+B,E)))
\]
degenerates at $E_1$
\end{lemma}

\begin{theorem}
\label{thm:gradedUnitaryVanishing}
Let $B$ be a smooth very ample divisor transveral to $D$, then for 
$m = 0, \cdots, n-1$
\[
  H^q(X, \Gr^W_m\DR^p(D,E,\nabla)\otimes O_X(B)) = 0
\]
for $p + q > n+1$.
\end{theorem}

\begin{proof}
We show first that 
\[
  H^q(X, W_0\DR^p(D, E, \nabla)\otimes O_X(B)) = 0
\]
for $p + q > n+1$. 

By Theorem \ref{thm:partialWeightFiltration}, we have the exact sequence
\[
  0 \rightarrow W^B_0\DR_X(D+B,E) 
    \rightarrow W_0\DR_X(D,E)\otimes O_X(B) 
    \rightarrow W_0\DR_B(B\cap D, E_B)\otimes O_X(B)
    \rightarrow 0
\]

Take cohomology sequence, we get
\[
    \cdots\rightarrow H^q(X, W^B_0(\Omega^p_X(\log D+B)\otimes E)) 
    \rightarrow H^q(X, W_0(\Omega^p_X(\log D)\otimes E)\otimes O_X(B))
    \rightarrow 
\]
\[  
    \rightarrow H^q(B, W_0(\Omega^p_B(\log B\cap D)\otimes E)\otimes O_X(B))
    \rightarrow \cdots 
\]
Therefore, it is enough to show

\textbf{Claim 1:} $H^q(X, W_0^B(\Omega^p_X(\log D + B)\otimes E)) = 0$ 
for $p + q > n$.

\textbf{Claim 2:} 
$H^q(B, W_0(\Omega^p_B(\log B\cap D)\otimes E_B))\otimes O_X(B)) = 0$
for $p + q > n$.

\textbf{Proof of Claim 1:}
Consider the maps
\[
  X - (B+D)\xrightarrow{f} X - B \xrightarrow{h} X
\]
Write $V^o$ for the restriction of $V$ on $X-B$. We have seen
in Lemma \ref{lemma:spectralSequenceDegWB0} that the spectral
sequence
\begin{align*}
  E_1^{p,q} = H^q(X, W^B_0(\Omega^p_X(\log D + B)\otimes E)) 
  => & \hypercohomology^{p+q}(X, \bcomplex) \\
  = & H^{p+q}(X, \derivedpushforward{h}\pushforward f V^o) \\
  = & H^{p+q}(X-B, \pushforward f V^o) 
\end{align*}

As $X - B$ is affine, it follows from Theorem \ref{thm:artinVanishing} that
\[
  H^q(X, W^B_0(\Omega^p_X(\log D + B)\otimes E)) = 0
\]
for $p + q > n$.

\textbf{End}

\textbf{Proof of Claim 2:} Induct on dimension of $X$ \textbf{End}

Therefore, it remains to show that if $X$ is a smooth projective curve,
then
\[
  H^1(X, W_0(\Omega_X(\log D)\otimes E)\otimes O_X(B)) = 0
\]
But for the curve case, 
\[
  W_0(\Omega_X(\log D)\otimes E)\otimes O_X(B) 
  = W^B_0(\Omega_X(\log D + B)\otimes E) 
\]
Therefore the result follows again from Theorem \ref{thm:artinVanishing}

To finish the rest of the proof, we use the identification from proposition
\ref{prop:gradedComplex}
\[
  (W_m/W_{m-1})\DR_X(D,E) \cong 
  W_0\DR(\tilde C_m, E_m, \nabla_m)[-m]
\]
and then apply the above argument to $\tilde D_m$.
\end{proof}
  
\begin{corollary}
For any $m \in \integer$, 
\[
  H^q(X, W_m(\Omega_X^p(\log D)\otimes E)\otimes O_X(B)) = 0
\]
for $p + q > \dim X$
\end{corollary}

\begin{corollary}
Let $L$ be an ample line bundle on $X$,  then
\[
  H^q(X, \Gr^W_.(\Omega^p_X(\log D)\otimes E)\otimes L) = 0
\]
for $p + q > n$.
\end{corollary}

\begin{proof}
Like in Theorem \ref{thm:gradedUnitaryVanishing}, it is enough to show
\[
  H^q(X, W_0\DR^p(D,E,\nabla)\otimes L) =0
\]
for $p + q > n$.  

Let $m$ be a large enough integer such that $\sheaftensor{L}{m}$ 
is very ample. Let $B$ be a smooth hyperplane divisor transversal to
$D$ so that 
\[
  L \cong O_X(B)
\]

Use the same idea from Corollary \ref{coro:unitaryVanishing}, 
we construct a cyclic cover of degree $m$ branched over $B$
\[
  \pi: X^{\prime} \rightarrow X
\]

To finish the proof, it remains to show
\[
  \pullback \pi W_0\DR_X(D,E) = W_0\DR(\tilde D, \tilde E, \tilde \nabla)
\]
But this is clear from the local description of $W_0$.
\end{proof}

\section{Appendix}
\subsection{Linear algebra}
\begin{theorem}
Let $U$ be an unitary matrix over $\complex$, then $U$ is diagonalizable.
\end{theorem}

\begin{theorem}
Let $A$ and $B$ be commuting diagonalizable $n\times n$ matrices over any field $k$, 
then $A$ and $B$ can 
be simultaneously diagolized.
\end{theorem}

\begin{proof}
Let $V$ be the vector space $k^n$.
It is enough to show that $A$ and $B$ share the same eigenvectors. 

\textbf{Claim 1:}
$A$ and $B$ share at least one eigenvector.\newline
\textbf{Proof of Claim 1:} 
Let $v$ be an eigenvector of $A$ with eigenvalue $\lambda$,
then
\[
  ABv = BAv = B\lambda v = \lambda Bv
\]
i.e. $Bv$ is also an eigenvector of $A$ with eigenvalue $\lambda$.

Let $W$ be the subspace spanned by
\[
  v, Bv,\cdots, B^nv
\]
Then, $W$ is invariant under $B$. As $V$ has a basis by eigenvectors of $B$, one
can choose a vector $w \in W$ which is an eigenvector of $B$. Then, from the 
construction of $W$, $w$ is also an eigenvector of $A$.
\textbf{End}

Let $w$ be as above, with $Bw = \mu w$;
Let $e_1,\cdots,e_n$ be the standard basis of $V$; 
Let $V^{\prime}$ be the subspace spanned by $e_1,\cdots,e_{n-1}$;
Let $\phi: V \rightarrow V$ be the linear map such that $\phi(e_n) = w$.

\begin{align*}
  \inverse{\phi}\circ A\circ \phi & = A^{\prime}\oplus \text{Diag}(\lambda) \\
  \inverse{\phi}\circ B\circ \phi & = B^{\prime}\oplus \text{Diag}(\mu) 
\end{align*} 
where $A^{\prime}$ and $B^{\prime}$ are $n-1\times n-1$ submatrices of $A$ and 
$B$, representing the restriction of $A$ and $B$ on $V^{\prime}$. 

Now, $A^{\prime}$ and $B^{\prime}$ are diagonlizable, and they commute, therefore, by 
inducting on the size of the matrix, we are done.   
\end{proof}

\newpage

\bibliography{_ref}
\bibliographystyle{plain}

\end{document}